\documentclass[12pt,a4paper]{article}
\usepackage[utf8x]{inputenc}
\usepackage{amsmath}
\usepackage{amsfonts}
\usepackage{amssymb}
\usepackage{amsthm}
\usepackage{mathrsfs}
\usepackage{fancyhdr}
\usepackage{graphicx}
\usepackage[usenames,x11names]{xcolor}
\usepackage{arabtex}
\usepackage{framed}
\usepackage[top=2cm,bottom=2cm,margin=2cm]{geometry}
\usepackage{cancel}
\usepackage{array}   
\usepackage{ulem}

\usepackage{hyperref}

\title{$q$-analogues of sums of consecutive powers of natural numbers and extended Carlitz $q$-Bernoulli numbers and polynomials}
\author{\sc Bakir FARHI \\
National Higher School of Mathematics \\
P.O.Box 75, Mahelma 16093, Sidi Abdellah (Algiers) \\
Algeria \\[1mm]
\href{mailto:bakir.farhi@nhsm.edu.dz}{\tt bakir.farhi@nhsm.edu.dz} \\[1mm]
\url{http://farhi.bakir.free.fr/}
}
\date{}

\let\up=\textsuperscript
\let\epsilon=\varepsilon

\def\R{{\mathbb R}}

\def\N{{\mathbb N}}
\def\Z{{\mathbb Z}}

\def\E{\mathscr{E}}

\def\deg{\mathrm{deg}}
\def\card{\mathrm{Card}}

\def\idem{\leavevmode\hbox to 10.6mm{\vrule height .63ex depth -.59ex
    width 10mm\hfill}}

\newcommand{\qbinom}{\genfrac{[}{]}{0pt}{}} 

\theoremstyle{plain}
\numberwithin{equation}{section}
\newtheorem{thm}{Theorem}[section]

\newtheorem{lemma}[thm]{Lemma}
\newtheorem{prop}[thm]{Proposition}
\newtheorem{coll}[thm]{Corollary}

\theoremstyle{definition}
\newtheorem{defi}[thm]{Definition}

\theoremstyle{remark}
\newtheorem{rmk}[thm]{Remark}

\newtheorem{expl}[thm]{Example}
\newtheorem{expls}[thm]{Examples}

\begin{document}
\maketitle

\begin{abstract}
In this paper, we investigate a specific class of $q$-polynomial sequences that serve as a $q$-analogue of the classical Appell sequences. This framework offers an elegant approach to revisiting classical results by Carlitz and, more interestingly, to establishing an important extension of the Carlitz $q$-Bernoulli polynomials and numbers. In addition, we establish explicit series representations for our extended Carlitz $q$-Bernoulli numbers and express them in terms of $q$-Stirling numbers of the second kind. This leads to a novel formula that explicitly connects the Carlitz $q$-Bernoulli numbers with the $q$-Stirling numbers of the second kind.  
\end{abstract}


\section{Introduction and Notation}\label{sec1}

Throughout this paper, we let $\N$ and $\N_0$ respectively denote the set of positive integers and the set of nonnegative integers. For $n \in \N_0$, we let ${(X)}_n$ denote the \textit{falling factorial} of $X$ to depth $n$, defined as
$$
{(X)}_n := X (X - 1) \cdots (X - n + 1) .
$$
The Stirling numbers of the second kind, denoted $S(n , k)$ ($n , k \in \N_0$, $n \geq k$), are then defined through the polynomial identity:
$$
X^n = \sum_{k = 0}^{n} S(n , k) {(X)}_k ~~~~~~~~~~ (\forall n \in \N_0) .
$$
The Bernoulli polynomials are denoted, as usual, by $B_n(X)$ and the Bernoulli numbers by $B_n$ ($n \in \N_0$). A modern definition of the Bernoulli polynomials and numbers employs their exponential generating functions, given by
$$
\frac{t}{e^t - 1} e^{X t} = \sum_{n = 0}^{\infty} B_n(X) \frac{t^n}{n!} ~~,~~  
\frac{t}{e^t - 1} = \sum_{n = 0}^{\infty} B_n \frac{t^n}{n!}  
$$
(so $B_n = B_n(0)$ for all $n \in \N_0$). As is well-known, the Bernoulli polynomials and numbers play a central role in various branches of mathematics, including number theory, mathematical analysis, and algebraic geometry. For a modern and exhaustive perspective, the reader is referred to the paper by Kouba \cite{kou}. An interesting and well-known formula expressing the Bernoulli numbers in terms of the Stirling numbers of the second kind is given by:
\begin{equation}\label{eqn3}
B_n = \sum_{k = 0}^{n} (-1)^k \frac{k!}{k + 1} S(n , k) ~~~~~~~~~~ (\forall n \in \N_0)
\end{equation}
(see e.g., \cite[Corollary 2.8]{far}).

It should be noted that the Bernoulli polynomial sequence ${\left(B_n(X)\right)}_{n \in \N_0}$ belongs to the broader class of Appell sequences (see \cite{cos,rom}). These are sequences of polynomials ${\left(P_n(X)\right)}_{n \in \N_0}$ characterized by the following properties: $P_0(X)$ is a nonzero constant polynomial, and for all $n \in \N$, the derivatives satisfy the relation $P_n'(X) = n P_{n - 1}(X)$. Using the Taylor expansion, an Appell polynomial sequence ${\left(P_n(X)\right)}_n$ can also be characterized by its general term, which takes the form
$$
P_n(X) = \sum_{k = 0}^{n} \binom{n}{k} p_k X^{n - k} ~~~~~~~~~~ (\forall n \in \N_0) ,
$$
where ${(p_k)}_{k \in \N_0}$ is a sequence of scalars with $p_0 \neq 0$.

Further, let $q$ be a positive real parameter. The $q$-analogue of $X$, whether treated as an indeterminate or a number, is defined by:
$$
[X] := \frac{q^X - 1}{q - 1} .
$$
For $k \in \N_0$, the $q$-analogue ${[X]}_k$ of ${(X)}_k$ is defined as
$$
[X]_k := [X] \cdot [X - 1] \cdots [X - k + 1] .
$$
For $n \in \N_0$, the $q$-analogue of $n!$ is given by
$$
[n]! := {[n]}_n = [n] \cdot [n - 1] \cdots [1] .
$$
Furthermore, for $n , k \in \N_0$ with $n \geq k$, the $q$-analogue $\qbinom{n}{k}$ of the binomial coefficient $\binom{n}{k}$ is defined as
$$
\qbinom{n}{k} := \frac{[n]!}{[k]! [n - k]!} = \frac{{[n]}_k}{[k]!} .
$$
These numbers appear, in particular, in the famous Gauss binomial formula:
\begin{equation}\label{eqn1}
\prod_{k = 0}^{n - 1} \left(x + q^k y\right) = \sum_{k = 0}^{n} q^{\frac{k (k - 1)}{2}} \qbinom{n}{k} x^{n - k} y^k ~~~~~~~~~~ (\forall n \in \N, \forall x , y \in \R) ,
\end{equation}
and, just like the classical binomial coefficients, they also have a combinatorial interpretation (see, e.g., \cite{kac}).

On the other hand, we let $\Delta$ denote the forward difference operator which acts linearly on $\R[X]$ by the formula:
$$
(\Delta P)(X) := P(X + 1) - P(X) ~~~~~~~~~~ (\forall P \in \R[X]) .
$$
It is well known and easy to verify that the $n$-fold composition of $\Delta$ is given by the formula:
\begin{equation}\label{eqn2}
(\Delta^n P)(X) = \sum_{k = 0}^{n} (-1)^{n - k} \binom{n}{k} P(X + k) ~~~~~~~~~~ (\forall P \in \R[X] , \forall n \in \N_0) .
\end{equation}

In this paper, we call a \textit{$q$-polynomial} any polynomial in $[X]$ with real coefficients. Equivalently, a $q$-polynomial is a polynomial in $q^X$ with real coefficients. The degree of a $q$-polynomial refers to its degree as a polynomial in $[X]$ (or equivalently in $q^X$). We denote by $\E$ the $\R$-vector space of all $q$-polynomials and by $\E_d$ ($d \in \N_0$) the $\R$-vector subspace of $\E$ consisting of $q$-polynomials of degree $\leq d$. An important example of $q$-polynomials arises naturally in the evaluation of sums of the form
$$
S_{n , r}(N) := \sum_{k = 0}^{N - 1} q^{r k} {[k]}^n
$$
($N , n \in \N_0$, $r \in \N$), which are $q$-analogues of the sum of powers $\sum_{k = 0}^{N - 1} k^n$. More interestingly, when expressed as a linear combination of the $q$-polynomials $q^{k N} {[N]}^{n - k}$ ($0 \leq k \leq n$), $S_{n , r}(N)$ leads to a $q$-analogue example of Appell polynomial sequences (see Theorem \ref{t2}). Perhaps it was by exploring this fact that Carlitz \cite{car} obtained important $q$-analogues of Bernoulli numbers and polynomials. By means of the symbolic calculus, Carlitz \cite{car} defined two real sequences ${(\eta_n)}_{n \in \N_0}$ and ${(\beta_n)}_{n \in \N_0}$ (depending on $q$) by:
\begin{align}
~ & \left\{\begin{array}{l}
\eta_0 = 1 , \eta_1 = 0 , \\
(q \eta + 1)^n = \eta^n ~~ (\forall n \geq 2)
\end{array}
\right. , \label{eq4} \\
~ & \left\{\begin{array}{l}
\beta_0 = 1 , \\
q (q \beta + 1)^n - \beta^n = \delta_{n , 1} ~~ (\forall n \in \N)
\end{array}
\right. , \label{eq5}
\end{align}
where $\delta_{i , j}$ denotes the Kronecker delta. Furthermore, it is noted that for all $n \in \N_0$:
\begin{equation}\label{eq6}
\beta_n = \eta_n + (q - 1) \eta_{n + 1} .
\end{equation}
Although both sequences are intriguing, the sequence ${(\beta_n)}_n$ is particularly noteworthy because its terms are well-defined at $q = 1$ (unlike ${(\eta_n)}_n$, for which only the first two terms are defined at $q = 1$). Moreover, specializing $q = 1$ in $\beta_n$ recovers the $n$\up{th} Bernoulli number $B_n$. These facts are far from trivial. To prove them, Carlitz was led to introduce $q$-analogues of the Stirling numbers of the second kind, $S_q(n , k) (n , k \in \N_0 , n \geq k)$, which he defined via the following identity of $q$-polynomials:
\begin{equation}\label{eq7}
{[X]}^n = \sum_{k = 0}^{n} q^{\frac{1}{2} k (k - 1)} S_q(n , k) {[X]}_k
\end{equation}
(for all $n \in \N_0$). Specializing $q = 1$ in $S_q(n , k)$ ($n , k \in \N_0$, $n \geq k$) recovers the classical Stirling numbers of the second kind $S(n , k)$. Using the $S_q(n , k)$'s, Carlitz \cite{car} established the following formula for $\beta_n$ ($n \in \N_0$):
\begin{equation}\label{eq8}
\beta_n = \sum_{k = 0}^{n} (-1)^k \frac{[k]!}{[k + 1]} S_q(n , k) ,
\end{equation}
which confirms that $\beta_n$ is always well-defined at $q = 1$, with value equal to the Bernoulli number $B_n$ (according to \eqref{eqn3}).

Carlitz also associated to the sequences ${(\eta_n)}_n$ and ${(\beta_n)}_n$ the $q$-polynomial sequences ${(\eta_n(X))}_n$ and ${(\beta_n(X))}_n$, given by:
\begin{align}
\eta_n(X) & = \sum_{k = 0}^{n} \binom{n}{k} \eta_k q^{k X} {[X]}^{n - k} , \label{eq9} \\
\beta_n(X) & = \sum_{k = 0}^{n} \binom{n}{k} \beta_k q^{k X} {[X]}^{n - k}
\end{align}
(for all $n \in \N_0$). The original sequences ${(\eta_n)}_n$ and ${(\beta_n)}_n$ are then respectively obtained as the values of ${(\eta_n(X))}_n$ and ${(\beta_n(X))}_n$ at $X = 0$. Especially, ${(\beta_n(X))}_n$ constitutes a $q$-analogue of the classical Bernoulli polynomial sequence ${(B_n(X))}_n$. That said, alternative definitions of the Carlitz $q$-Bernoulli numbers and polynomials exist, such as those based on generating functions or $p$-adic $q$-integrals on $\Z_p$. For example, Koblitz \cite{kob} constructed $q$-analogues of the $p$-adic Dirichlet $L$-series that interpolate the Carlitz $q$-Bernoulli numbers.

In this work, we first incorporate the Carlitz \(q\)-polynomial sequences \({(\eta_n(X))}_n\) and \({(\beta_n(X))}_n\) into a broader class of \(q\)-polynomial sequences, which we call ``Carlitz-type \(q\)-polynomial sequences''. We establish a fundamental theorem showing that this class of \(q\)-polynomial sequences serves as a \(q\)-analogue of the classical Appell polynomial sequences. Next, We verify that the closed forms of the $q$-analogues of power sums of consecutive natural numbers can essentially be expressed in terms of Carlitz-type $q$-polynomial sequences. Furthermore, we show that a specific subclass of Carlitz-type $q$-polynomial sequences (including \({(\eta_n(X))}_n\) and \({(\beta_n(X))}_n\)) can be generated using simple symbolic formulas, such as those established by Carlitz in \cite{car}. We then establish several properties of Carlitz-type \(q\)-polynomial sequences and leverage these properties to extend \({(\eta_n(X))}_n\) and \({(\beta_n(X))}_n\) into \({(\beta_n^{(r)}(X))}_n\) (\(r \in \mathbb{N}_0\)), which provide important \(q\)-analogues of the classical Bernoulli polynomials for $r \geq 1$. Additionally, we represent the numbers \(\beta_n^{(r)} := \beta_n^{(r)}(0)\) as series and express them in terms of the \(q\)-Stirling numbers of the second kind. Finally, we derive a novel formula connecting the Carlitz \(q\)-Bernoulli numbers with the \(q\)-Stirling numbers of the second kind.

\section{The results and the proofs}\label{sec2}

\subsection{Carlitz-type $q$-polynomial sequences}\label{subsec1}

We first introduce the following definition:

\begin{defi}
Let ${(T_n(X))}_{n \in \N_0}$ be a $q$-polynomial sequence. We say that ${(T_n(X))}_n$ is \textit{Carlitz-type} if its general term $T_n(X)$ ($n \in \N_0$) has the form:
$$
T_n(X) = (q - 1)^{- n} \sum_{k = 0}^{n} (-1)^{n - k} \binom{n}{k} a_k q^{k X} ,
$$
where ${(a_k)}_{k \in \N_0}$ is a real sequence. In this case, ${(a_k)}_k$ is called \textit{the associated sequence} of the Carlitz-type $q$-polynomial sequence ${(T_n(X))}_n$.
\end{defi}

An important characterization of the Carlitz-type $q$-polynomial sequences is given by the following fundamental theorem:

\begin{thm}\label{t1}
Let ${(T_n(X))}_{n \in \N_0}$ be a $q$-polynomial sequence, and for all $n \in \N_0$, set $t_n := T_n(0)$. Then ${(T_n(X))}_n$ is Carlitz-type if and only if we have for all $n \in \N_0$:
$$
T_n(X) = \sum_{k = 0}^{n} \binom{n}{k} t_k q^{k X} {[X]}^{n - k} ,
$$
that is (symbolically):
$$
T_n(X) = \left(q^X t + [X]\right)^n .
$$
\end{thm}

To prove this theorem, we need to use the inversion formula from the following lemma:

\begin{lemma}[see e.g, \cite{com,rio}]\label{l1}
Let ${(u_n)}_{n \in \N_0}$ and ${(v_n)}_{n \in \N_0}$ be two real sequences. Then the two following identities are equivalent:
\begin{alignat}{3}
u_n & = \sum_{k = 0}^{n} \binom{n}{k} v_k ~~~~~~~~~~& (\forall n \in \N_0) , \notag \\[1mm]
v_n & = \sum_{k = 0}^{n} (-1)^{n - k} \binom{n}{k} u_k ~~~~~~~~~~& (\forall n \in \N_0) . \tag*{\qedsymbol}
\end{alignat}
\end{lemma}

\begin{proof}[Proof of Theorem \ref{t1}]
Suppose that ${(T_n(X))}_n$ is Carlitz-type. So there exists a real sequence ${(a_n)}_{n \in \N_0}$ for which we have for all $n \in \N_0$:
$$
T_n(X) = (q - 1)^{- n} \sum_{k = 0}^{n} (-1)^{n - k} \binom{n}{k} a_k q^{k X} .
$$
In particular, we have for all $n \in \N_0$:
$$
t_n := T_n(0) = (q - 1)^{- n} \sum_{k = 0}^{n} (-1)^{n - k} \binom{n}{k} a_k ,
$$
that is
$$
(q - 1)^n t_n = \sum_{k = 0}^{n} (-1)^{n - k} \binom{n}{k} a_k .
$$
By inverting this last formula (using Lemma \ref{l1}), we derive that for all $n \in \N_0$:
$$
a_n = \sum_{k = 0}^{n} \binom{n}{k} (q - 1)^k t_k .
$$
Then, by inserting this into the expression of $T_n(X)$, we obtain that for all $n \in \N_0$:
\begin{align*}
T_n(X) & = (q - 1)^{- n} \sum_{k = 0}^{n} (-1)^{n - k} \binom{n}{k} \left(\sum_{i = 0}^{k} \binom{k}{i} (q - 1)^i t_i\right) q^{k X} \\
& = (q - 1)^{- n} \sum_{k = 0}^{n} \sum_{i = 0}^{k} (-1)^{n - k} \binom{n}{k} \binom{k}{i} (q - 1)^i t_i q^{k X} .
\end{align*}
By inverting the summations and remarking that $\binom{n}{k} \binom{k}{i} = \binom{n}{i} \binom{n - i}{k - i}$ (for all $n \in \N_0$, $k \in \{0 , 1 , \dots , n\}$, and $i \in \{0 , 1 , \dots , k\}$), we derive that for all $n \in \N_0$:
\begin{align*}
T_n(X) & = (q - 1)^{- n} \sum_{0 \leq i \leq n} \sum_{i \leq k \leq n} (-1)^{n - k} \binom{n}{i} \binom{n - i}{k - i} (q - 1)^i t_i q^{k X} \\
& \hspace*{-1cm}= (q - 1)^{- n} \sum_{0 \leq i \leq n} \sum_{0 \leq j \leq n - i} (-1)^{n - i - j} \binom{n}{i} \binom{n - i}{j} (q - 1)^i t_i q^{(i + j) X} ~~~ (\text{where we set $j = k - i$}) \\
& \hspace*{-1cm}= (q - 1)^{- n} \sum_{0 \leq i \leq n} \binom{n}{i} (q - 1)^i t_i q^{i X} \sum_{0 \leq j \leq n - i} \binom{n - i}{j} q^{j X} (-1)^{n - i - j} .
\end{align*}
But according to the binomial formula, we have for all $n \in \N_0$ and all $i \in \{0 , 1 , \dots , n\}$:
$$
\sum_{0 \leq j \leq n - i} \binom{n - i}{j} q^{j X} (-1)^{n - i - j} = \left(q^X - 1\right)^{n - i} = (q - 1)^{n - i} {[X]}^{n - i} .
$$
By inserting this into the last obtained expression for $T_n(X)$, we finally derive that for all $n \in \N_0$:
$$
T_n(X) = \sum_{0 \leq i \leq n} \binom{n}{i} t_i q^{i X} {[X]}^{n - i} ,
$$
as required.

Conversely, suppose that we have for all $n \in \N_0$:
$$
T_n(X) = \sum_{k = 0}^{n} \binom{n}{k} t_k q^{k X} {[X]}^{n - k} ,
$$
and define for all $n \in \N_0$:
\begin{align}
a_n & := \sum_{k = 0}^{n} \binom{n}{k} (q - 1)^k t_k \label{eq10} \\[-6mm]
\intertext{and} \notag \\[-12mm]
S_n(X) & := (q - 1)^{- n} \sum_{k = 0}^{n} (-1)^{n - k} \binom{n}{k} a_k q^{k X} . \notag
\end{align}
Equivalently, ${(S_n(X))}_{n \in \N_0}$ is a Carlitz-type $q$-polynomial sequence, with associated sequence ${(a_n)}_{n \in \N_0}$. So, according to the first part of this proof, we have for all $n \in \N_0$:
\begin{equation}\label{eq11}
S_n(X) = \sum_{k = 0}^{n} \binom{n}{k} s_k q^{k X} {[X]}^{n - k} ,
\end{equation}
where
\begin{equation}\label{eq12}
s_n := S_n(0) = (q - 1)^{- n} \sum_{k = 0}^{n} (-1)^{n - k} \binom{n}{k} a_k .
\end{equation}
But the inversion of Formula \eqref{eq10} (using Lemma \ref{l1}) gives for all $n \in \N_0$:
$$
(q - 1)^n t_n = \sum_{k = 0}^{n} (-1)^{n - k} \binom{n}{k} a_k ,
$$
that is
$$
t_n = (q - 1)^{- n} \sum_{k = 0}^{n} (-1)^{n - k} \binom{n}{k} a_k = s_n
$$
(according to \eqref{eq12}). Hence we have for all $n \in \N_0$:
\begin{align*}
T_n(X) & = \sum_{k = 0}^{n} \binom{n}{k} t_k q^{k X} {[X]}^{n - k} \\
& = \sum_{k = 0}^{n} \binom{n}{k} s_k q^{k X} {[X]}^{n - k} \\
& = S_n(X) ~~~~~~~~~~ (\text{according to \eqref{eq11}}) .
\end{align*}
Consequently ${(T_n(X))}_n = {(S_n(X))}_n$ is Carlitz-type, as required. This complete this proof.
\end{proof}

\begin{rmk}
From the perspective of Theorem \ref{t1}, the Carlitz-type $q$-polynomial sequences can be regarded as a $q$-analogue of the Appell polynomial sequences.
\end{rmk}

\begin{expls}\label{expls1}
In \cite{car}, Carlitz introduced two important $q$-polynomial sequences, which are ${(\eta_n(X))}_{n \in \N_0}$ and ${(\beta_n(X))}_{n \in \N_0}$, defined for all $n \in \N_0$ by:
\begin{align}
\eta_n(X) & := (q - 1)^{- n} \sum_{k = 0}^{n} (-1)^{n - k} \binom{n}{k} \frac{k}{[k]} q^{k X} , \label{eq15} \\
\beta_n(X) & := (q - 1)^{- n} \sum_{k = 0}^{n} (-1)^{n - k} \binom{n}{k} \frac{k + 1}{[k + 1]} q^{k X} \label{eq16}
\end{align}
(where we take by convention $\frac{k}{[k]} = 1$ for $k = 0$). From their definitions, it is clear that ${(\eta_n(X))}_n$ and ${(\beta_n(X))}_n$ are both Carlitz-type with associated real sequences ${(\frac{k}{[k]})}_k$ and ${(\frac{k + 1}{[k + 1]})}_k$ respectively. In addition, Carlitz \cite{car} explored the real sequences constituted of the values of ${(\eta_n(X))}_n$ and ${(\beta_n(X))}_n$ at $X = 0$; that is the sequences ${(\eta_n)}_{n \in \N_0}$ and ${(\beta_n)}_{n \in \N_0}$, defined by:
$$
\eta_n := \eta_n(0) ~~\text{and}~~ \beta_n := \beta_n(0) ~~~~~~~~~~ (\forall n \in \N_0) .
$$
So, according to Theorem \ref{t1}, we have important alternative expressions for $\eta_n(X)$ and $\beta_n(X)$ ($n \in \N_0$), which are
\begin{align}
\eta_n(X) & = \sum_{k = 0}^{n} \binom{n}{k} \eta_k q^{k X} {[X]}^{n - k} , \label{eq17} \\
\beta_n(X) & = \sum_{k = 0}^{n} \binom{n}{k} \beta_k q^{k X} {[X]}^{n - k} . \label{eq18}
\end{align}
Using the symbolic calculus, Formulas \eqref{eq17} and \eqref{eq18} can be written more simply as:
\begin{align}
\eta_n(X) & = \left(q^X \eta + [X]\right)^n , \label{eq19} \\
\beta_n(X) & = \left(q^X \beta + [X]\right)^n
\end{align}
(for all $n \in \N_0$).
\end{expls}

We now see that Carlitz-type $q$-polynomial sequences naturally emerge when deriving closed-form expressions for $q$-analogues of power sums of consecutive positive integers. We begin with the following theorem:

\begin{thm}\label{t2}
Let $r$ be a positive integer. For all nonnegative integer $n$ and all positive integer $N$, we have
$$
\sum_{k = 0}^{N - 1} q^{r k} {[k]}^n = q^{r N} S_n(N) - S_n(0) ,
$$
where ${(S_k(X))}_{k \in \N_0}$ is the Carlitz-type $q$-polynomial sequence associated to the real sequence of general term $a_k = \frac{1}{q - 1} \frac{1}{[k + r]}$ ($\forall k \in \N_0$).
\end{thm}

\begin{proof}
Let $n \in \N_0$ and $N \in \N$ be fixed. We have
\begin{align*}
\sum_{k = 0}^{N - 1} q^{r k} {[k]}^n & = \sum_{k = 0}^{N - 1} q^{r k} \left(\frac{q^k - 1}{q - 1}\right)^n \\
& = (q - 1)^{- n} \sum_{k = 0}^{N - 1} q^{r k} \left(q^k - 1\right)^n \\
& = (q - 1)^{- n} \sum_{k = 0}^{N - 1} q^{r k} \sum_{i = 0}^{n} \binom{n}{i} q^{k i} (-1)^{n - i} ~~~~~~~~~~ (\text{by the binomial formula}) \\
& = (q - 1)^{- n} \sum_{i = 0}^{n} (-1)^{n - i} \binom{n}{i} \sum_{k = 0}^{N - 1} q^{(i + r) k} \\
& = (q - 1)^{- n} \sum_{i = 0}^{n} (-1)^{n - i} \binom{n}{i} \frac{q^{(i + r)N} - 1}{q^{i + r} - 1} \\
& = (q - 1)^{- n} \sum_{i = 0}^{n} (-1)^{n - i} \binom{n}{i} \frac{1}{(q - 1) [i + r]} \left(q^{(i + r) N} - 1\right) .
\end{align*}
Setting for all $n \in \N_0$:
$$
S_n(X) := (q - 1)^{- n} \sum_{i = 0}^{n} (-1)^{n - i} \binom{n}{i} \frac{1}{(q - 1) [i + r]} q^{i X}
$$
(so ${(S_n(X))}_n$ is the Carlitz-type $q$-polynomial sequence associated to the real sequence of general term $a_i = \frac{1}{(q - 1) [i + r]}$), we conclude that:
$$
\sum_{n = 0}^{N - 1} q^{r k} {[k]}^n = q^{r N} S_n(N) - S_n(0) ,
$$
as required.
\end{proof}

The following generalization of Theorem \ref{t2} aims to provide a larger class of real sequences, for which the associated Carlitz-type $q$-polynomial sequences naturally arise when evaluating $q$-analogues of power sums of consecutive positive integers.

\begin{thm}\label{t3}    
Let $r$ be a positive integer and $d$ be a nonnegative integer such that $r \geq d$. For all integer $n \geq d$ and all positive integer $N$, we have
$$
\sum_{k = 0}^{N - 1} q^{r k} {[k]}^{n - d} = \frac{q^{(r - d) N} T_n(N) - T_n(0)}{n (n - 1) \cdots (n - d + 1)} ,
$$
where ${(T_k(X))}_{k \in \N_0}$ is the Carlitz-type $q$-polynomial sequence associated to the real sequence of general term
$$
b_k = (q - 1)^{d - 1} \frac{k (k - 1) \cdots (k - d + 1)}{[k + r - d]} ~~~~~~~~~~ (\forall k \in \N_0) .
$$
\end{thm}

\begin{proof}
Let $n$ and $N$ be fixed integers such that $n \geq d$ and $N \geq 1$. According to Theorem \ref{t2}, we have that
$$
\sum_{k = 0}^{N - 1} q^{r k} {[k]}^{n - d} = q^{r N} S_{n - d}(N) - S_{n - d}(0) ,
$$
where ${(S_k(X))}_{k \in \N_0}$ is the Carlitz-type $q$-polynomial sequence associated to the real sequence of general term $a_k = \frac{1}{(q - 1) [k + r]}$ ($\forall k \in \N_0$). By introducing the $q$-polynomial sequence ${(T_k(X))}_{k \in \N_0}$ defined by:
$$
T_k(X) = \begin{cases}
0 & \text{if } k < d \\
k (k - 1) \cdots (k - d + 1) q^{d X} S_{k - d}(X) & \text{else}
\end{cases} ~~~~~~~~~~ (\forall k \in \N_0) ,
$$
the closed form of the sum of $q$-powers above is transformed to
$$
\sum_{k = 0}^{N - 1} q^{r k} {[k]}^{n - d} = \frac{q^{(r - d) N} T_n(N) - T_n(0)}{n (n - 1) \cdots (n - d + 1)} .
$$
So it remains to verify that ${(T_k(X))}_k$ is Carlitz-type with the associated sequence that given by the theorem. For an integer $k \geq d$, we have
$$
T_k(X) := k (k - 1) \cdots (k - d + 1) q^{d X} S_{k - d}(X) \\
$$
\begin{align*}
& = k (k - 1) \cdots (k - d + 1) q^{d X} \cdot (q - 1)^{- k + d} \sum_{i = 0}^{k - d} (-1)^{k - d - i} \binom{k - d}{i} \frac{1}{(q - 1) [i + r]} q^{i X} \\
& = k (k - 1) \cdots (k - d + 1) q^{d X} (q - 1)^{- k + d} \sum_{j = d}^{k} (-1)^{k - j} \binom{k - d}{j - d} \frac{1}{(q - 1) [j + r - d]} q^{(j - d) X}
\end{align*}
(by setting $j = i + d$). But since $k (k - 1) \cdots (k - d + 1) \binom{k - d}{j - d} = j (j - 1) \cdots (j - d + 1) \binom{k}{j}$ (for all integers $k , j \geq d$), it follows that for all integer $k \geq d$:
\begin{align*}
T_k(X) & = (q - 1)^{- k} \sum_{j = d}^{k} (-1)^{k - j} \binom{k}{j} (q - 1)^{d - 1} \frac{j (j - 1) \cdots (j - d + 1)}{[j + r - d]} q^{j X} \\
& = (q - 1)^{- k} \sum_{j = 0}^{k} (-1)^{k - j} \binom{k}{j} (q - 1)^{d - 1} \frac{j (j - 1) \cdots (j - d + 1)}{[j + r - d]} q^{j X} .
\end{align*}
By remarking that the last formula for $T_k(X)$ is obviously valid for $0 \leq k < d$, we conclude that ${(T_k(X))}_k$ is Carlitz-type and its associated real sequence has the general term \linebreak $b_j = (q - 1)^{d - 1} \frac{j (j - 1) \cdots (j - d + 1)}{[j + r - d]}$ ($\forall j \in \N_0$). The theorem is thus proved.
\end{proof}

\begin{rmk}\label{rmk1}
It is important to note that the general term of the real sequence ${(b_k)}_{k \in \N_0}$ associated to the $q$-polynomial sequence in Theorem \ref{t3} has the form $\frac{P(k)}{[k + s]}$, with $P$ is a real polynomial and $s$ is a nonnegative integer. This fact will be used below for establishing some important symbolic formulas.
\end{rmk}

In the following theorem, we will establish symbolic formulas generating the values at $0$ of some particular Carlitz-type $q$-polynomial sequences.

\begin{thm}\label{t4}
Let $r$ and $d$ be two nonnegative integers and $P$ be a real polynomial of degree $d$. Let also ${(T_n(X))}_{n \in \N_0}$ be the Carlitz-type $q$-polynomial sequence whose associated real sequence is ${\left(\frac{P(k)}{[k + r]}\right)}_{k \in \N_0}$ and ${(t_n)}_{n \in \N_0}$ be the real sequence defined by $t_n := T_n(0)$ ($\forall n \in \N_0$). Then ${(t_n)}_n$ satisfies the symbolic formula:
$$
q^r \left(q t + 1\right)^n = t^n ~~~~~~~~~~ (\forall n > d) .
$$
\end{thm}

\begin{proof}
Let $n > d$ be a fixed integer. From the definition of $T_n(X)$, we have that
$$
T_n(X) = (q - 1)^{- n} \sum_{k = 0}^{n} (-1)^{n - k} \binom{n}{k} \frac{P(k)}{[k + r]} q^{k X} .
$$
Thus
\begin{align*}
q^r T_n(1) - T_n(0) & = (q - 1)^{- n} \sum_{k = 0}^{n} (-1)^{n - k} \binom{n}{k} \frac{P(k)}{[k + r]} \left(q^{k + r} - 1\right) \\
& = (q - 1)^{- n + 1} \sum_{k = 0}^{n} (-1)^{n - k} \binom{n}{k} P(k) ~~~~~~~~ (\text{since } q^{k + r} - 1 = (q - 1) [k + r]) \\
& = (q - 1)^{- n + 1} \left(\Delta^n P\right)(0) ~~~~~~~~~~~~~~~~~~~~~~~ (\text{according to \eqref{eqn2}}) \\
& = 0 ~~~~~~~~~~ (\text{since } n > \deg{P} = d) .
\end{align*}
Consequently, we have
\begin{equation}\label{eq13}
q^r T_n(1) = T_n(0) = t_n .
\end{equation}
On the other hand, we have (according to Theorem \ref{t1}) the symbolic formula:
$$
T_n(X) = \left(q^X t + [X]\right)^n .
$$
In particular, we have symbolically
$$
T_n(1) = \left(q t + 1\right)^n .
$$
By inserting this into \eqref{eq13}, it finally follows that we have symbolically
$$
q^r \left(q t + 1\right)^n = t^n ,
$$
as required.
\end{proof}

\begin{expls}~
\begin{enumerate}
\item Since the general term of the associated real sequences of the Carlitz-type $q$-polynomial sequences ${(\eta_n(X))}_{n \in \N_0}$ and ${(\beta_n(X))}_{n \in \N_0}$ are respectively $\frac{k}{[k]}$ and $\frac{k + 1}{[k + 1]}$ then (according to Theorem \ref{t4}) the real sequences ${(\eta_n)}_{n \in \N_0}$ and ${(\beta_n)}_{n \in \N_0}$ satisfy the symbolic formulas:
\begin{equation*}
\begin{split}
\left(q \eta + 1\right)^n & = \eta^n \\
q \left(q \beta + 1\right)^n & = \beta^n
\end{split} ~~~~~~~~~~ (\forall n > 1) .
\end{equation*}
\item Let us put ourselves in the situation of Theorem \ref{t3} and set $t_k := T_k(0)$ (for all $k \in \N_0$). Since the general term of the associated real sequence of the Carlitz-type $q$-polynomial sequence ${(T_k(X))}_k$ has the form $\frac{P(k)}{[k + r - d]}$, where $P$ is a real polynomial of degree $d$, then (according to Theorem \ref{t4}) the sequence ${(t_k)}_k$ satisfy the symbolic formula:
$$
q^{r - d} \left(q t + 1\right)^k = t^k ~~~~~~~~~~ (\forall k > d) .
$$
\end{enumerate}
\end{expls}

In the following proposition, we provide some important properties of Carlitz-type $q$-polynomial sequences. For simplicity, given a real or functional sequence ${(f_n)}_{n \in \N_0}$, we denote (with a slight abuse of notation) by ${(n f_{n - 1})}_{n \in \N_0}$ the sequence whose terms are defined as 
\[\begin{cases} 0 & \text{if } n = 0 , \\ n f_{n - 1} & \text{otherwise} \end{cases} .\]  

\begin{prop}\label{p1}
Let ${(T_n(X))}_{n \in \N_0}$ be a Carlitz-type $q$-polynomial sequence and ${(a_k)}_{k \in \N_0}$ be its associated real sequence. Then the following properties hold.
\begin{enumerate}
\item The $q$-polynomial sequence ${\left(q^X n T_{n - 1}(X)\right)}_{n \in \N_0}$ is Carlitz-type with associated real sequence ${\left((q - 1) k a_{k - 1}\right)}_{k \in \N_0}$.
\item If $T_0(X) = 0_{\E}$ (i.e., $a_0 = 0$) then the $q$-polynomial sequence ${\left(\frac{q^{- X} T_{n + 1}(X)}{n + 1}\right)}_{n \in \N_0}$ is Carlitz-type with associated real sequence ${\left(\frac{1}{q - 1} \frac{a_{k + 1}}{k + 1}\right)}_{k \in \N_0}$.
\item The sequence ${\left(q^{- X}\left(T_n(X) + (q - 1) T_{n + 1}(X)\right)\right)}_{n \in \N_0}$ is a Carlitz-type $q$-polynomial sequence with associated real sequence ${(a_{k + 1})}_{k \in \N_0}$.
\item The $q$-polynomial sequence ${\big(n\left(T_{n - 1}(X) + (q - 1) T_n(X)\right)\big)}_{n \in \N_0}$ is Carlitz-type with associated real sequence ${\left((q - 1) k a_k\right)}_{k \in \N_0}$.
\end{enumerate}
\end{prop}

\begin{proof}
By definition, we have for all $n \in \N_0$:
\begin{equation}\label{eq14}
T_n(X) = (q - 1)^{- n} \sum_{k = 0}^{n} (-1)^{n - k} \binom{n}{k} a_k q^{k X} .
\end{equation}
\textbullet{} Let us show the 1\up{st} item of the proposition. Using \eqref{eq14}, we have for all $n \in \N_0$:
\begin{align*}
q^X n T_{n - 1}(X) & = (q - 1)^{- n + 1} \sum_{k = 0}^{n - 1} (-1)^{n - 1 - k} n \binom{n - 1}{k} a_k q^{(k + 1) X} \\
& = (q - 1)^{- n} \sum_{k = 1}^{n} (-1)^{n - k} n \binom{n - 1}{k - 1} (q - 1) a_{k - 1} q^{k X} .
\end{align*}
Remarking that $n \binom{n - 1}{k - 1} = k \binom{n}{k}$ (for all $1 \leq k \leq n$), it follows that for all $n \in \N_0$:
$$
q^X n T_{n - 1}(X) = (q - 1)^{- n} \sum_{k = 0}^{n} (-1)^{n - k} \binom{n}{k} \left((q - 1) k a_{k - 1}\right) q^{k X} ,
$$
showing that the $q$-polynomial sequence ${\left(q^X n T_{n - 1}(X)\right)}_{n \in \N_0}$ is Carlitz-type with associated real sequence ${\left((q - 1) k a_{k - 1}\right)}_{k \in \N_0}$. The first item of the proposition is proved. \\
\textbullet{} Next, let us show the 2\up{nd} item of the proposition. So, suppose that $T_0(X) = 0_{\E}$, that is $a_0 = 0$. Using \eqref{eq14}, we have for all $n \in \N_0$:
\begin{align*}
q^{- X} \frac{T_{n + 1}(X)}{n + 1} & = (q - 1)^{- n - 1} \sum_{k = 1}^{n + 1} (-1)^{n + 1 - k} \frac{1}{n + 1} \binom{n + 1}{k} a_k q^{(k - 1) X} \\
& = (q - 1)^{- n} \sum_{k = 0}^{n} (-1)^{n - k} \frac{1}{n + 1} \binom{n + 1}{k + 1} \frac{1}{q - 1} a_{k + 1} q^{k X} .
\end{align*}
Remarking that $\frac{1}{n + 1} \binom{n + 1}{k + 1} = \frac{1}{k + 1} \binom{n}{k}$ (for all $0 \leq k \leq n$), it follows that for all $n \in \N_0$:
$$
q^{- X} \frac{T_{n + 1}(X)}{n + 1} = (q - 1)^{- n} \sum_{k = 0}^{n} (-1)^{n - k} \binom{n}{k} \left(\frac{1}{q - 1} \frac{a_{k + 1}}{k + 1}\right) q^{k X} ,
$$
showing that the $q$-polynomial sequence ${\left(q^{- X} \frac{T_{n + 1}(X)}{n + 1}\right)}_{n \in \N_0}$ is Carlitz-type with associated real sequence ${\left(\frac{1}{q - 1} \frac{a_{k + 1}}{k + 1}\right)}_{k \in \N_0}$. The second item of the proposition is proved. \\
\textbullet{} Now, let us show the 3\up{rd} item of the proposition. Using \eqref{eq14}, we have for all $n \in \N_0$:
\begin{align*}
q^{- X} \left(T_n(X) + (q - 1) T_{n + 1}(X)\right) & = (q - 1)^{- n} \sum_{k = 0}^{n} (-1)^{n - k} \binom{n}{k} a_k q^{(k - 1) X} \\
& \hspace*{2cm} + (q - 1)^{- n} \sum_{k = 0}^{n + 1} (-1)^{n + 1 - k} \binom{n + 1}{k} a_k q^{(k - 1) X} \\
& = (q - 1)^{- n} \sum_{k = 0}^{n + 1} (-1)^{n - k + 1} \left(\binom{n + 1}{k} - \binom{n}{k}\right) a_k q^{(k - 1) X} .
\end{align*}
But since $\binom{n + 1}{k} - \binom{n}{k} = \begin{cases}
0 & \text{if } k = 0 \\
\binom{n}{k - 1} & \text{if } k \geq 1 
\end{cases}$ (for all $0 \leq k \leq n + 1$), it follows that:
\begin{align*}
q^{- X} \left(T_n(X) + (q - 1) T_{n + 1}(X)\right) & = (q - 1)^{- n} \sum_{k = 1}^{n + 1} (-1)^{n - k + 1} \binom{n}{k - 1} a_k q^{(k - 1) X} \\
& = (q - 1)^{- n} \sum_{k = 0}^{n} (-1)^{n - k} \binom{n}{k} a_{k + 1} q^{k X} ,
\end{align*}
showing that the sequence ${\left(q^{- X} \left(T_n(X) + (q - 1) T_{n + 1}(X)\right)\right)}_{n \in \N_0}$ is $q$-polynomial of Carlitz-type with associated real sequence ${(a_{k + 1})}_{k \in \N_0}$. The third item of the proposition is proved. \\
\textbullet{} Finally, the 4\up{th} item of the proposition follows by successively applying its 3\up{rd} and 1\up{st} items, which have already been proven. This completes the proof.
\end{proof}

\begin{expl}
Since the real sequence associated to ${(\eta_n(X))}_{n \in \N_0}$ (which is a Carlitz-type $q$-polynomial sequence) is ${(\frac{k}{[k]})}_{k \in \N_0}$ then (according to Item 3 of Proposition \ref{p1}) the sequence ${\left(q^{- X} \left(\eta_n(X) + (q - 1) \eta_{n + 1}(X)\right)\right)}_{n \in \N_0}$ is a Carlitz-type $q$-polynomial sequence with associated real sequence ${\left(\frac{k + 1}{[k + 1]}\right)}_{k \in \N_0}$. Hence ${\left(q^{- X} \left(\eta_n(X) + (q - 1) \eta_{n + 1}(X)\right)\right)}$ coincides with ${\left(\beta_n(X)\right)}_n$; that is
\begin{equation}\label{eq25}
\beta_n(X) = q^{-X} \left(\eta_n(X) + (q - 1) \eta_{n + 1}(X)\right) ~~~~~~~~~~ (\forall n \in \N_0) .
\end{equation}
\end{expl}

We now derive from Theorem \ref{t3} formulas for $q$-analogues of sums of powers of positive integers, which involve the $q$-polynomial sequences ${(\eta_n(X))}_{n \in \N_0}$ and ${(\beta_n(X))}_{n \in \N_0}$ introduced by Carlitz in \cite{car}. The following corollaries are obtained:

\begin{coll}\label{c2}
For all positive integers $n$ and $N$, we have
$$
\sum_{k = 0}^{N - 1} q^k {[k]}^{n - 1} = \frac{\eta_n(N) - \eta_n}{n} .
$$ 
\end{coll}

\begin{proof}
Take in Theorem \ref{t3}: $r = d = 1$.
\end{proof}

\begin{coll}\label{c3}
For all positive integers $n$ and $N$, we have
$$
n \sum_{k = 0}^{N - 1} q^{2 k} {[k]}^{n - 1} + (q - 1) \sum_{k = 0}^{N - 1} q^k {[k]}^n = q^N \beta_n(N) - \beta_n .
$$
\end{coll}

\begin{proof}
Let $n$ and $N$ be fixed positive integers. By applying Theorem \ref{t3} for the couples $(r , d) = (2 , 1)$ and $(r , d) = (1 , 0)$, we respectively obtain the two following formulas:
\begin{align*}
\sum_{k = 0}^{N - 1} q^{2 k} {[k]}^{n - 1} & = \frac{q^N T_n(N) - T_n(0)}{n} , \\
\sum_{k = 0}^{N - 1} q^k {[k]}^n & = q^N S_n(N) - S_n(0) ,
\end{align*}
where ${(T_k(X))}_{k \in \N_0}$ and ${(S_k(X))}_{k \in \N_0}$ represent the Carlitz-type $q$-polynomial sequences, associated respectively to the real sequences ${\left(\frac{k}{[k + 1]}\right)}_{k \in \N_0}$ and ${\left(\frac{1}{q - 1} \frac{1}{[k + 1]}\right)}_{k \in \N_0}$. By combining the two above formulas, we derive that:
$$
n \sum_{k = 0}^{N - 1} q^{2 k} {[k]}^{n - 1} + (q - 1) \sum_{k = 0}^{N - 1} q^k {[k]}^n = q^N U_n(N) - U_n(0) ,
$$
where
$$
U_k(X) := T_k(X) + (q - 1) S_k(X) ~~~~~~~~~~ (\forall k \in \N_0) .
$$
Next, from the fact that ${(T_k(X))}_k$ and ${(S_k(X))}_k$ are Carlitz-type $q$-polynomial sequences whose associated real sequences are respectively ${\left(\frac{k}{[k + 1]}\right)}_k$ and ${\left(\frac{1}{q - 1} \frac{1}{[k + 1]}\right)}_k$, we derive that ${(U_k(X))}_k$ is also a Carlitz-type $q$-polynomial sequence, with an associated real sequence having the general term:
$$
\frac{k}{[k + 1]} + (q - 1) \left(\frac{1}{q - 1} \frac{1}{[k + 1]}\right) = \frac{k + 1}{[k + 1]} ~~~~~~~~~~ (\forall k \in \N_0) .
$$
Consequently, ${(U_k(X))}_k$ is nothing else than the $q$-polynomial sequence ${(\beta_k(X))}_k$ of Carlitz $q$-Bernoulli polynomials. This completes the proof.
\end{proof}

\subsection{Extended Carlitz $q$-Bernoulli numbers and polynomials}\label{subsec2}

\subsubsection{Definitions and basic results}

We begin by introducing the following definition:

\begin{defi}
We define ${\left(\beta_n^{(r)}(X)\right)}_{r , n \in \N_0}$ recursively by:
\begin{alignat}{3}
\beta_n^{(0)}(X) & := \eta_n(X) &~~~~~~~~~~ (\forall n \in \N_0) , \label{eq20} \\
\beta_n^{(r + 1)}(X) & := q^{- X} \left(\beta_n^{(r)}(X) + (q - 1) \beta_{n + 1}^{(r)}(X)\right) &~~~~~~~~~~ (\forall r , n \in \N_0) . \label{eq21}
\end{alignat}
Besides, we define ${\left(\beta_n^{(r)}\right)}_{r , n \in \N_0}$ by:
\begin{equation}\label{eq22}
\beta_n^{(r)} := \beta_n^{(r)}(0) ~~~~~~~~~~ (\forall r , n \in \N_0) .
\end{equation}
So, it follows from \eqref{eq20} and \eqref{eq21} that:
\begin{alignat}{3}
\beta_n^{(0)} & = \eta_n &~~~~~~~~~~ (\forall n \in \N_0) , \label{eq23} \\
\beta_n^{(r + 1)} & = \beta_n^{(r)} + (q - 1) \beta_{n + 1}^{(r)} &~~~~~~~~~~ (\forall r , n \in \N_0) . \label{eq24}
\end{alignat}
\end{defi}

Since ${\left(\beta_n^{(0)}(X)\right)}_{n \in \N_0} = {\left(\eta_n(X)\right)}_{n \in \N_0}$ is a Carlitz-type $q$-polynomial sequence (see Examples \ref{expls1}) then (according to Item 3 of Proposition \ref{p1} and a simple induction on $r$) all the sequences ${\left(\beta_n^{(r)}(X)\right)}_{n \in \N_0}$ ($r \in \N_0$) are Carlitz-type $q$-polynomial sequences. 

\begin{defi}\label{defi1}
Giving $r \in \N_0$, the $q$-polynomials $\beta_n^{(r)}(X)$ ($n \in \N_0$) are called \textit{the extended Carlitz $q$-Bernoulli polynomials of order $r$}, and the numbers $\beta_n^{(r)}$ ($n \in \N_0$) are called \textit{the extended Carlitz $q$-Bernoulli numbers of order $r$}.
\end{defi}

Taking $r = 0$ in \eqref{eq21}, we find in particular that for all $n \in \N_0$:
$$
\beta_n^{(1)}(X) = q^{- X} \left(\eta_n(X) + (q - 1) \eta_{n + 1}(X)\right) = \beta_n(X)
$$
(according to \eqref{eq25}). So, the Carlitz $q$-Bernoulli polynomials of order $1$ are nothing else than the Carlitz $q$-Bernoulli polynomials $\beta_n(X)$ (introduced by Carlitz in \cite{car}). Consequently, the Carlitz $q$-Bernoulli numbers of order $1$ are nothing else than the Carlitz $q$-Bernoulli numbers $\beta_n$ of \cite{car}. Furthermore, Carlitz \cite{car} showed that the $q$-polynomials $\beta_n(X) = \beta_n^{(1)}(X)$ ($n \in \N_0$) are well-defined for $q = 1$ and coincide with the Bernoulli polynomials when specializing $q$ to $1$; that is $\beta_n^{(1)}(X)\vert_{q = 1} = B_n(X)$ for all $n \in \N_0$. Consequently, in view of \eqref{eq21}, the $q$-polynomials $\beta_n^{(r)}(X)$ ($r \in \N$, $n \in \N_0$) are also well-defined for $q = 1$, and for all $r \in \N$ and $n \in \N_0$, we have:
$$
\beta_n^{(r + 1)}(X)\vert_{q = 1} = \beta_n^{(r)}(X)\vert_{q = 1} .
$$
By induction on $r$, it follows that for all $r \in \N$ and $n \in \N_0$:
$$
\beta_n^{(r)}(X)\vert_{q = 1} = \beta_n^{(1)}(X)\vert_{q = 1} = B_n(X) .
$$
This demonstrates that for all $r \in \N$, the $q$-polynomials $\beta_n^{(r)}(X)$ ($n \in \N_0$) can serve as a $q$-analogue of the classical Bernoulli polynomials. In view of \eqref{eq22}, it follows that for all $r \in \N$, the numbers $\beta_n^{(r)}$ ($n \in \N_0$) can also serve as a $q$-analogue of the classical Bernoulli numbers. This justify the terminology introduced in Definition \ref{defi1}.

The following proposition gathers several formulas for the Carlitz $q$-Bernoulli polynomials and numbers of a given order, which can be directly derived from the results of \S\ref{subsec1}.

\begin{prop}\label{p2}
For all $r , n \in \N_0$, we have that:
\begin{align}
\beta_n^{(r)}(X) & = (q - 1)^{- n} \sum_{k = 0}^{n} (-1)^{n - k} \binom{n}{k} \frac{k + r}{[k + r]} q^{k X} , \label{eq26} \\
\beta_n^{(r)} & = (q - 1)^{- n} \sum_{k = 0}^{n} (-1)^{n - k} \binom{n}{k} \frac{k + r}{[k + r]} , \label{eq27} \\
\beta_n^{(r)}(X) & = \sum_{k = 0}^{n} \binom{n}{k} \beta_k^{(r)} q^{k X} {[X]}^{n - k} . \label{eq28}
\end{align}
In addition, for a given $r \in \N_0$, the Carlitz $q$-Bernoulli polynomials and numbers of order $r$ satisfy the following symbolic formulas:
\begin{alignat}{3}
q^r \left(q \beta^{(r)} + 1\right)^n & = \left(\beta^{(r)}\right)^n & ~~~~~~~~~~ (\forall n \geq 2) , \label{eq29} \\
\beta_n^{(r)}(X) & = \left(q^X \beta^{(r)} + [X]\right)^n & ~~~~~~~~~~ (\forall n \in \N_0) . \label{eq30} 
\end{alignat}
\end{prop}

\begin{proof}
Using Item 3 of Proposition \ref{p1} together with Formula \eqref{eq21}, we immediately show by induction that for all $r \in \N_0$, the associated real sequence of the Carlitz-type $q$-polynomial sequence ${\left(\beta_n^{(r)}(X)\right)}_{n \in \N_0}$ is ${\left(\frac{k + r}{[k + r]}\right)}_{k \in \N_0}$. Formula \eqref{eq26} of the proposition then follows. Next, Formula \eqref{eq27} of the proposition follows by taking $X = 0$ in Formula \eqref{eq26}, which we have just proven. Formulas \eqref{eq28} and \eqref{eq30} of the proposition immediately follow from Theorem \ref{t1}. Finally, Formula \eqref{eq29} of the proposition immediately follows from Theorem \ref{t4}. This completes this proof.
\end{proof}

We now see, for a general $r \in \N_0$, how the $q$-polynomials $\beta_n^{(r)}(X)$ arise in $q$-analogues of sums of powers of consecutive positive integers. We have the following theorem providing a generalization of Corollaries \ref{c2} and \ref{c3}.

\begin{thm}\label{t5}
Let $r$ be a nonnegative integer. Then we have for all positive integers $n$ and $N$:
$$
n \sum_{k = 0}^{N - 1} q^{(r + 1) k} {[k]}^{n - 1} + (q - 1) r \sum_{k = 0}^{N - 1} q^{r k} {[k]}^n = q^{r N} \beta_n^{(r)}(N) - \beta_n^{(r)} .
$$
\end{thm}

\begin{proof}
Let $n$ and $N$ be fixed positive integers. By applying Theorem \ref{t3} for the couples $(r , 0)$ and $(r + 1 , 1)$ (instead of $(r , d)$), we respectively obtain the two following formulas:
\begin{align*}
\sum_{k = 0}^{N - 1} q^{r k} {[k]}^n & = q^{r N} T_n(N) - T_n(0) , \\
\sum_{k = 0}^{N - 1} q^{(r + 1) k} {[k]}^{n - 1} & = \frac{q^{r N} S_n(N) - S_n(0)}{n} ,
\end{align*}
where ${\left(T_k(X)\right)}_{k \in \N_0}$ and ${\left(S_k(X)\right)}_{k \in \N_0}$ are the Carlitz-type $q$-polynomial sequences associated respectively to the real sequences of general term:
$$
a_k = \frac{1}{q - 1} \frac{1}{[k + r]} ~~~~\text{and}~~~~ b_k = \frac{k}{[k + r]} ~~~~~~~~~~ (\forall k \in \N_0) .
$$
By combining the above formulas, we derive that:
$$
n \sum_{k = 0}^{N - 1} q^{(r + 1) k} {[k]}^{n - 1} + (q - 1) r \sum_{k = 0}^{N - 1} q^{rk} {[k]}^n = q^{r N} U_n(N) - U_n(0) ,
$$
where
$$
U_k(X) := S_k(X) + (q - 1) r T_k(X) ~~~~~~~~~~ (\forall k \in \N_0) .
$$
On the other hand, from the fact that ${\left(T_k(X)\right)}_k$ and ${\left(S_k(X)\right)}_k$ are Carlitz-type $q$-polynomial sequences whose associated real sequences are respectively ${\left(\frac{1}{q - 1} \frac{1}{[k + r]}\right)}_k$ and ${\left(\frac{k}{[k + r]}\right)}_k$, we derive that ${\left(U_k(X)\right)}_k$ is also a Carlitz-type $q$-polynomial sequence, with an associated real sequence having the general term:
$$
\frac{k}{[k + r]} + (q - 1) r \left(\frac{1}{q - 1} \frac{1}{[k + r]}\right) = \frac{k + r}{[k + r]} ~~~~~~~~~~ (\forall k \in \N_0) .
$$
Hence ${\left(U_k(X)\right)}_k$ coincides with ${\left(\beta_k^{(r)}(X)\right)}_k$, concluding to the required result.
\end{proof}

\subsubsection{Representation of the $\beta_n^{(r)}$'s as series}

We now turn our attention to the representation of the extended Carlitz $q$-Bernoulli numbers $\beta_n^{(r)}$ as numerical series. We have the following theorem:

\begin{thm}\label{t6}
Suppose that $\vert{q}\vert < 1$. Then we have for all $r \in \N_0$ and $n \in \N$:
$$
\beta_n^{(r)} = \begin{cases}
\displaystyle - n \sum_{k = 0}^{+ \infty} q^k {[k]}^{n - 1} + (1 - q)^{- n} & \text{if } r = 0 \\[3mm]
\displaystyle - n \sum_{k = 0}^{+ \infty} q^{(r + 1) k} {[k]}^{n - 1} + r (1 - q) \sum_{k = 0}^{+ \infty} q^{r k} {[k]}^n & \text{else}
\end{cases} .
$$
\end{thm}

\begin{proof}
The result follows by letting $N \to + \infty$ in the formula of Theorem \ref{t5}. However, before doing so, we must evaluate $\lim_{N \rightarrow + \infty} q^{r N} \beta_n^{(r)}(N)$ (for given $r \in \N_0$, $n \in \N$). To achieve this, we use Formula \eqref{eq26} of Proposition \ref{p2}. Let $r \in \N_0$ and $n , N \in \N$ be fixed. According to Formula \eqref{eq26} of Proposition \ref{p2}, we have that:
\begin{align*}
\beta_n^{(r)}(N) & = (q - 1)^{- n} \sum_{k = 0}^{n} (-1)^{n - k} \binom{n}{k} \frac{k + r}{[k + r]} q^{k N} \\
& = (q - 1)^{- n} \left[(-1)^n \frac{r}{[r]} + \sum_{k = 1}^{n} (-1)^{n - k} \binom{n}{k} \frac{k + r}{[k + r]} q^{k N}\right] .
\end{align*}
Since $\vert{q}\vert < 1$ by hypothesis, we derive that:
$$
\lim_{N \rightarrow + \infty} \beta_n^{(r)}(N) = (1 - q)^{- n} \frac{r}{[r]} .
$$
Thus
$$
\lim_{N \rightarrow + \infty} q^{r N} \beta_n^{(r)}(N) = \begin{cases}
\lim_{N \rightarrow + \infty} \beta_n^{(r)}(N) = (1 - q)^{- n} & \text{if } r = 0 \\
0 & \text{else}
\end{cases} .
$$
Finally, by letting $N \rightarrow + \infty$ in the formula of Theorem \ref{t5} and incorporating the result we just obtained, we get
$$
n \sum_{k = 0}^{+ \infty} q^{(r + 1) k} {[k]}^{n - 1} + (q - 1) r \sum_{k = 0}^{+ \infty} q^{r k} {[k]}^n = \begin{cases}
(1 - q)^{- n} - \beta_n^{(r)} & \text{if } r = 0 \\
- \beta_n^{(r)} & \text{else} 
\end{cases} ,
$$
implying the desired result of the theorem and completes the proof.
\end{proof}

\begin{rmk}\label{rmk2}
For $r \in \N$, the formula of Theorem \ref{t6} can be written more simply as:
\begin{equation}\label{eq31}
\beta_n^{(r)} = \sum_{k = 0}^{+ \infty} \left(r q^{r k} - (n + r) q^{(r + 1) k}\right) {[k]}^{n - 1} .
\end{equation}
Taking in particular $r = 1$, we derive the following important representation of the Carlitz $q$-Bernoulli numbers as series:
\begin{equation}\label{eq32}
\beta_n = \sum_{k = 0}^{+ \infty} \left(q^k - (n + 1) q^{2 k}\right) {[k]}^{n - 1}
\end{equation}
(valid for $\vert{q}\vert < 1$ and $n \in \N$).
\end{rmk}

From Remark \ref{rmk2} and the fact that the extended Carlitz $q$-Bernoulli numbers of any positive order are $q$-analogues of the classical Bernoulli numbers, we immediately derive the following corollary:

\begin{coll}\label{c4}
For all $r , n \in \N$, we have that
$$
B_n = \lim_{q \stackrel{<}{\to 1}} \sum_{k = 0}^{+ \infty} \left(r q^{r k} - (n + r) q^{(r + 1) k}\right) {[k]}^{n - 1} .
$$
\end{coll}

\begin{rmk}\label{rmk3}
If, in the formula of Corollary \ref{c4}, we allow the limit to be taken inside the sum and also permit the manipulation of divergent series, we obtain, for all $n \in \N$:
\[
B_n = \sum_{k = 0}^{+\infty} -n k^{n-1} = -n \sum_{k = 1}^{+\infty} k^{n-1} = -n \zeta(1-n).
\]
This provides a well-known and valid formula for the classical Bernoulli numbers, where $\zeta$ denotes the Riemann zeta function. 
\end{rmk}

\subsubsection{Expression of the $\beta_n^{(r)}$'s in terms of $q$-Stirling numbers of the second kind}

The following theorem provides a generalization of Formula \eqref{eq8} of Carlitz, which express the Carlitz $q$-Bernoulli numbers in terms of $q$-Stirling numbers of the second kind.

\begin{thm}\label{t7}
For all $r \in \N$ and $n \in \N_0$, we have
$$
\beta_n^{(r)} = \sum_{k = 0}^{n} \varphi_r(q , k) S_q(n , k) ,
$$
where
\begin{align*}
\varphi_r(q , k) & = (-1)^k [k]! \sum_{i = 0}^{r - 1} q^{i k} \frac{{[r - 1]}_i}{{[k + r]}_{i + 1}} \\
& = (-1)^k [k]! [r - 1]! q^{- r + 1} \sum_{i = 0}^{r - 1} \frac{(q - 1)^i q^{\frac{1}{2} i (i + 1) - i (r - 1)}}{[r - 1 - i]!} \frac{1}{[k + i + 1]}
\end{align*}
(for all $k \in \N$).
\end{thm} 

The proof of this theorem needs some preparations which are presented below. We begin with the following proposition:

\begin{prop}\label{p3}
For all $r , N \in \N$ and $k \in \N_0$, we have
$$
\sum_{i = 0}^{N - 1} q^{i r} {[i]}_k = \left(\sum_{\ell = 0}^{r - 1} q^{(r - \ell - 1) N + (\ell + 1) k} \frac{{[r - 1]}_{\ell}}{{[k + r]}_{\ell + 1}}\right) {[N]}_{k + 1} .
$$
\end{prop}

\begin{proof}
Let $k \in \N_0$ and $N \in \N$ be fixed. For all $r \in \N$, we have
\begin{align*}
\Delta\left(q^{(r - 1) X} {[X]}_{k + 1}\right) & = q^{(r - 1) (X + 1)} {[X + 1]}_{k + 1} - q^{(r - 1) X} {[X]}_{k + 1} \\
& = q^{(r - 1) (X + 1)} [X + 1] [X] \cdots [X - k + 1] - q^{(r - 1) X} [X] [X - 1] \cdots [X - k] \\
& = \left(q^{(r - 1) (X + 1)} [X + 1] - q^{(r - 1) X} [X - k]\right) {[X]}_k .
\end{align*}
But since
\begin{align*}
q^{(r - 1) (X + 1)} [X + 1] - q^{(r - 1) X} [X - k] & = q^{(r - 1) (X + 1)} \frac{q^{X + 1} - 1}{q - 1} - q^{(r - 1) X} \frac{q^{X - k} - 1}{q - 1} \\
& = \frac{q^{r (X + 1)} - q^{(r - 1) (X + 1)} - q^{r X - k} + q^{(r - 1) X}}{q - 1} \\
& = \frac{q^{r X - k} \left(q^{k + r} - 1\right) - q^{(r - 1) X} \left(q^{r - 1} - 1\right)}{q - 1} \\
& = q^{r X - k} [k + r] - q^{(r - 1) X} [r - 1] ,
\end{align*}
it follows that:
$$
\Delta\left(q^{(r - 1) X} {[X]}_{k + 1}\right) = \left(q^{r X - k} [k + r] - q^{(r - 1) X} [r - 1]\right) {[X]}_k .
$$
By multiplying both sides of this last equality by $\frac{q^k}{[k + r]}$ and then summing both sides of the resulting equality from $X = 0$ to $N - 1$, we get (after simplifying and rearranging)
$$
\sum_{X = 0}^{N - 1} q^{r X} {[X]}_k = \frac{q^{(r - 1) N + k}}{[k + r]} {[N]}_{k + 1} + q^k \frac{[r - 1]}{[k + r]} \sum_{X = 0}^{N - 1} q^{(r - 1) X} {[X]}_k .
$$
By reiterating this last formula $r$ times, we get
\begin{align*}
\sum_{X = 0}^{N - 1} q^{r X} {[X]}_k & = q^{(r - 1) N + k} \frac{1}{[k + r]} {[N]}_{k + 1} + q^{(r - 2) N + 2 k} \frac{[r - 1]}{[k + r] [k + r - 1]} {[N]}_{k + 1} \\
& \hspace*{0.5cm} + q^{(r - 3) N + 3 k} \frac{{[r - 1]}_2}{{[k + r]}_3} {[N]}_{k + 1} + \dots + q^{r k} \frac{{[r - 1]}_{r - 1}}{{[k + r]}_r} {[N]}_{k + 1} + q^{r k} \frac{{[r - 1]}_r}{{[k + r]}_r} \sum_{X = 0}^{N - 1} {[X]}_k \\
& = \left(\sum_{\ell = 0}^{r - 1} q^{(r - \ell - 1) N + (\ell + 1) k} \frac{{[r - 1]}_{\ell}}{{[k + r]}_{\ell + 1}}\right) {[N]}_{k + 1} ~~~~~~~~~~ (\text{since } {[r - 1]}_r = 0) .  
\end{align*}
The proposition is proved.
\end{proof}

In what follows, for $r , N \in \N$ and $k \in \N_0$, we set
$$
f_r(q , k , N) := q^{\frac{k (k - 1)}{2}} \sum_{i = 0}^{N - 1} q^{i r} {[i]}_k ,
$$
which is $q$-polynomial in $N$ (according to Proposition \ref{p3}). We let also $f_r(q , k , X)$ denote the unique $q$-polynomial interpolating $f_r(q , k , N)$ at the positive integers $X = N$, and $g_r(q , k)$ denote the limit:
$$
g_r(q , k) := \lim_{X \to 0} \frac{1}{[X]} f_r(q , k , X) .
$$

A first useful expression for $g_r(q , k)$ is directly derived from Proposition \ref{p3}. We have the following corollary:

\begin{coll}\label{c5}
For all $r \in \N$ and $k \in \N_0$, we have
$$
g_r(q , k) = (-1)^k [k]! \left(\sum_{i = 0}^{r - 1} q^{i k} \frac{{[r - 1]}_i}{{[k + r]}_{i + 1}}\right) .
$$
\end{coll}

\begin{proof}
Let $r \in \N$ and $k \in \N_0$ be fixed. According to Proposition \ref{p3}, we have for all $N \in \N$:
$$
f_r(q , k , N) = q^{\frac{k (k - 1)}{2}} \left(\sum_{\ell = 0}^{r - 1} q^{(r - \ell - 1) N + (\ell + 1) k} \frac{{[r - 1]}_{\ell}}{{[k + r]}_{\ell + 1}}\right) {[N]}_{k + 1} .
$$
Hence
$$
f_r(q , k , X) = q^{\frac{k (k - 1)}{2}} \left(\sum_{\ell = 0}^{r - 1} q^{(r - \ell - 1) X + (\ell + 1) k} \frac{{[r - 1]}_{\ell}}{{[k + r]}_{\ell + 1}}\right) {[X]}_{k + 1} .
$$
Consequently, we have
\begin{align*}
g_r(q , k) & := \lim_{X \to 0} \frac{1}{[X]} f_r(q , k , X) \\
& = \lim_{X \to 0} q^{\frac{k (k - 1)}{2}} \left(\sum_{\ell = 0}^{r - 1} q^{(r - \ell - 1) X + (\ell + 1) k} \frac{{[r - 1]}_{\ell}}{{[k + r]}_{\ell + 1}}\right) {[X - 1]}_k \\
& = q^{\frac{k (k + 1)}{2}} \left(\sum_{\ell = 0}^{r - 1} q^{\ell k} \frac{{[r - 1]}_{\ell}}{{[k + r]}_{\ell + 1}}\right) {[-1]}_k \\
& = (-1)^k [k]! \sum_{\ell = 0}^{r - 1} q^{\ell k} \frac{{[r - 1]}_{\ell}}{{[k + r]}_{\ell + 1}}
\end{align*}
(since ${[-1]}_k = (-1)^k q^{- \frac{k (k + 1)}{2}} [k]!$). The corollary is proved.
\end{proof}

An equally interesting formula for $g_r(q , k)$ ($r \in \N , k \in \N_0$) can be derived by decomposing the $q$-rational function $\frac{q^{i X}}{{[X]}_{i + 1}}$ ($0 \leq i \leq r - 1$) into $q$-partial fractions. This decomposition is given by the following proposition:

\begin{prop}\label{p4}
For all $k \in \N_0$, we have
$$
\frac{q^{k X}}{{[X]}_{k + 1}} = \frac{1}{[k]!} \sum_{i = 0}^{k} \left((-1)^{k - i} \qbinom{k}{i} q^{\frac{k^2 + k}{2} + \frac{i^2 - i}{2}} \frac{1}{[X - i]}\right) .
$$
\end{prop}

\begin{proof}
Let $k \in \N_0$ be fixed. We will first establish the existence and uniqueness of real numbers $\alpha_{k , i}(q)$ ($0 \leq i \leq k$) such that the following holds:
\begin{equation}\label{eq33}
\frac{q^{k X}}{{[X]}_{k + 1}} = \sum_{i = 0}^{k} \frac{\alpha_{k , i}(q)}{[X - i]} .
\end{equation}
We have:
\begin{align}
\eqref{eq33} & \Longleftrightarrow q^{k X} = \sum_{i = 0}^{k} \alpha_{k , i}(q) \frac{{[X]}_{k + 1}}{[X - i]} \notag \\
& \Longleftrightarrow q^{k X} = \sum_{i = 0}^{k} \alpha_{k , i}(q) \left(\prod_{\begin{subarray}{c}
0 \leq j \leq k \\
j \neq i
\end{subarray}} [X - j]\right) . \label{eq34}
\end{align}
Consider the family $\mathscr{B}$ of $(k + 1)$ $q$-polynomials of degree $k$ given by:
$$
\mathscr{B} := \left\{\prod_{\begin{subarray}{c}
0 \leq j \leq k \\
j \neq i
\end{subarray}} [X - j]\right\}_{0 \leq i \leq k} .
$$
It can be observed that any $k$ arbitrary $q$-polynomials from $\mathscr{B}$ always share a common root that is not a root of the remaining $q$-polynomial in $\mathscr{B}$. Therefore, no $q$-polynomial in $\mathscr{B}$ can be expressed as a linear combination of the others, which shows that $\mathscr{B}$ is a linearly independent family in $\E_k$. Consequently, since $\card{\mathscr{B}} = \dim{\E_k} = k + 1$, $\mathscr{B}$ is a basis of $\E_k$. It follows that the $q$-polynomial $q^{k X}$ (belonging to $\E_k$) can be uniquely expressed as a linear combination of the $q$-polynomials in $\mathscr{B}$. This establishes the existence and uniqueness of real numbers $\alpha_{k , i}(q)$ ($0 \leq i \leq k$) satisfying \eqref{eq34} (equivalently \eqref{eq33}).

Let us now determine the real numbers $\alpha_{k , i}(q)$ ($0 \leq i \leq k$). For a given $\ell \in \{0 , 1 , \dots , k\}$, by multiplying both sides of \eqref{eq33} by $[X - \ell]$ and then substituting $X = \ell$ into the resulting identity, we obtain:
$$
\frac{q^{k \ell}}{[\ell] [\ell - 1] \cdots [1] \cdot [-1] [-2] \cdots [\ell - k]} = \alpha_{k , \ell}(q) .
$$
But since $[\ell] [\ell - 1] \cdots [1] = [\ell]!$ and $[-1] [-2] \cdots [\ell - k] = (-1)^{k - \ell} q^{- \frac{(k - \ell) (k - \ell + 1)}{2}} [k - \ell]!$ (using the property $[- X] = - q^{- X} [X]$), it follows that:
\begin{align*}
\alpha_{k , \ell}(q) & = \frac{(-1)^{k - \ell} q^{\frac{(k - \ell) (k - \ell + 1)}{2} + k \ell}}{[\ell]! [k - \ell]!} \\
& = \frac{1}{[k]!} (-1)^{k - \ell} \qbinom{k}{\ell} q^{\frac{k^2 + k}{2} + \frac{\ell^2 - \ell}{2}} , 
\end{align*}
as required. This completes the proof.
\end{proof}

Using the formula from Proposition \ref{p4}, the formula of Corollary \ref{c5} for $g_r(q , k)$ can be transformed into another, more interesting form for studying the limiting case as $q \to 1$. We have the following proposition:

\begin{prop}\label{p5}
For all $r \in \N$ and $k \in \N_0$, we have
$$
g_r(q , k) = (-1)^k [k]! [r - 1]! q^{- r + 1} \sum_{i = 0}^{r - 1} \dfrac{(q - 1)^i q^{\frac{1}{2} i (i + 1) - i (r - 1)}}{[r - 1 - i]!} \frac{1}{[k + i + 1]} .
$$
\end{prop}

\begin{proof}
Let $r \in \N$ and $k \in \N_0$ be fixed. By successively applying Corollary \ref{c5} and Proposition \ref{p4}, we obtain
\begin{align*}
g_r(q , k) & = (-1)^k [k]! \sum_{i = 0}^{r - 1} q^{- i r} {[r - 1]}_i \left\{\dfrac{q^{i (k + r)}}{{[k + r]}_{i + 1}}\right\} \\
& = (-1)^k [k]! \sum_{i = 0}^{r - 1} q^{- i r} {[r - 1]}_i \frac{1}{[i]!} \sum_{j = 0}^{i} (-1)^{i - j} \qbinom{i}{j} q^{\frac{i^2 + i}{2} + \frac{j^2 - j}{2}} \frac{1}{[k + r - j]} \\
& = (-1)^k [k]! \sum_{\begin{subarray}{l}
0 \leq i \leq r - 1 \\
0 \leq j \leq i
\end{subarray}} q^{- i r} \frac{{[r - 1]}_i}{[i]!} \qbinom{i}{j} (-1)^{i - j} q^{\frac{i^2 + i}{2} + \frac{j^2 - j}{2}} \frac{1}{[k + r - j]} .
\end{align*}
By inverting the summations over $i$ and $j$ and noting that $\frac{{[r - 1]}_i}{[i]!} \qbinom{i}{j} = \qbinom{r - 1}{j} \qbinom{r - 1 - j}{i - j}$ and that $q^{- i r} \cdot q^{\frac{i^2 + i}{2} + \frac{j^2 - j}{2}} = q^{\frac{(i - j)^2 +(i - j)}{2} - i (r - j)}$ (for all $0 \leq i \leq r - 1$ and $0 \leq j \leq i$), we derive that:
$$
g_r(q , k) = (-1)^k [k]! \sum_{0 \leq j \leq r - 1} \left(\sum_{i = j}^{r - 1} \qbinom{r - 1 - j}{i - j} (-1)^{i - j} q^{\frac{(i - j)^2 + (i - j)}{2} - i (r - j)}\right) \qbinom{r - 1}{j} \frac{1}{[k + r - j]} .
$$
Then, performing the change of index $\ell = i - j$ in the inner summation gives
$$
g_r(q , k) = (-1)^k [k]! \sum_{0 \leq j \leq r - 1} \left(\sum_{\ell = 0}^{r - 1 - j} q^{\frac{\ell (\ell - 1)}{2}} \qbinom{r - 1 - j}{\ell} \left(- q^{1 - r + j}\right)^{\ell}\right) q^{j^2 - j r} \qbinom{r - 1}{j} \frac{1}{[k + r - j]} .
$$
But according to the Gauss binomial formula \eqref{eqn1}, we have for all $j \in \{0 , 1 , \dots , r - 1\}$:
\begin{align*}
\sum_{\ell = 0}^{r - 1 - j} q^{\frac{\ell (\ell - 1)}{2}} \qbinom{r - 1 - j}{\ell} \left(- q^{1 - r + j}\right)^{\ell} & = \left(1 - q^{1 - r + j}\right) \left(1 - q^{2 - r + j}\right) \cdots \left(1 - q^{-1}\right) \\
& = q^{1 - r + j} \left(q^{r - 1 - j} - 1\right) \cdot q^{2 - r + j} \left(q^{r - 2 - j} - 1\right) \cdots q^{-1} (q - 1) \\
& = q^{- \frac{(r - 1 - j) (r - j)}{2}} (q - 1)^{r - 1 - j} [r - 1 - j]! .
\end{align*}
Substituting this into the last obtained expression for $g_r(q, k)$, we get
$$
g_r(q , k) = (-1)^k [k]! \sum_{j = 0}^{r - 1} q^{- \frac{(r - 1 - j) (r - j)}{2}} (q - 1)^{r - 1 - j} [r - 1 - j]! q^{j^2 - j r} \qbinom{r - 1}{j} \frac{1}{[k + r - j]} .
$$
Finally, changing the index $i = r - 1 - j$ and rearranging yields
$$
g_r(q , k) = (-1)^k [k]! [r - 1]! q^{- r + 1} \sum_{i = 0}^{r - 1} \dfrac{(q - 1)^i q^{\frac{1}{2} i (i + 1) - i (r - 1)}}{[r - 1 - i]!} \frac{1}{[k + i + 1]} ,
$$
as required.
\end{proof}

We are now ready to prove Theorem \ref{t7}.

\begin{proof}[Proof of Theorem \ref{t7}]
Let $r \in \N$ and $n \in \N_0$ be fixed. In view of Corollary \ref{c5} and Proposition \ref{p5}, we need to show that:
$$
\beta_n^{(r)} = \sum_{k = 0}^{n} g_r(q , k) S_q(n , k) .
$$
To achieve this, consider $N \in \N$ and evaluate the sum
$$
\sum_{i = 0}^{N - 1} q^{r i} [i]^n
$$
in two different ways. On the one hand, from Theorem \ref{t2}, we have
$$
\sum_{i = 0}^{N - 1} q^{r i} {[i]}^n = q^{r N} S_n(N) - S_n(0) ,
$$
where ${(S_k(X))}_{k \in \N_0}$ is the Carlitz-type $q$-polynomial sequence associated to the real sequence of general term $a_k := \frac{1}{q - 1} \frac{1}{[k + r]}$ ($\forall k \in \N_0$). So, putting $s_k := S_k(0)$ ($\forall k \in \N_0$) and using the formula of Theorem \ref{t1} for $S_k(X)$, we derive that:
\begin{align}
\sum_{i = 0}^{N - 1} q^{r i} {[i]}^n & = q^{r N} \sum_{k = 0}^{n} \binom{n}{k} s_k q^{k N} {[N]}^{n - k} - s_n \notag \\
& = \sum_{k = 0}^{n - 1} \binom{n}{k} s_k q^{(k + r) N} {[N]}^{n - k} + \left(q^{(n + r) N} - 1\right) s_n . \label{eq35}
\end{align}
On the other hand, by using Formula \eqref{eq7}, we have that:
\begin{align}
\sum_{i = 0}^{N - 1} q^{r i} {[i]}^n & = \sum_{i = 0}^{N - 1} q^{r i} \left(\sum_{k = 0}^{n} q^{\frac{1}{2} k (k - 1)} S_q(n , k) {[i]}_k\right) \notag \\
& = \sum_{k = 0}^{n} S_q(n , k) \left(q^{\frac{1}{2} k (k - 1)} \sum_{i = 0}^{N - 1} q^{r i} {[i]}_k\right) \notag \\
& = \sum_{k = 0}^{n} S_q(n , k) f_r(q , k , N) . \label{eq36}
\end{align}
Comparing \eqref{eq35} and \eqref{eq36}, we derive the identity:
$$
\sum_{k = 0}^{n - 1} \binom{n}{k} s_k q^{(k + r) N} {[N]}^{n - k} + \left(q^{(n + r) N} - 1\right) s_n = \sum_{k = 0}^{n} S_q(n , k) f_r(q , k , N) .
$$
Since this identity consists of $q$-polynomials in $N$ and it is valid for all $N \in \N$ then
it extends as a $q$-polynomial identity. Namely, we have
$$
\sum_{k = 0}^{n - 1} \binom{n}{k} s_k q^{(k + r) X} {[X]}^{n - k} + \left(q^{(n + r) X} - 1\right) s_n = \sum_{k = 0}^{n} S_q(n , k) f_r(q , k , X) .
$$
Dividing through by $[X]$ and letting $X \to 0$, we derive
\begin{equation}\label{eq37}
n s_{n - 1} + (q - 1) (n + r) s_n = \sum_{k = 0}^{n} g_r(q , k) S_q(n , k) .
\end{equation}
Now, setting for all $k \in \N_0$:
\begin{align*}
T_k(X) & := k S_{k - 1}(X) + (q - 1) (k + r) S_k(X) \\
& = (q - 1) r S_k(X) + k \left(S_{k - 1}(X) + (q - 1) S_k(X)\right) ,
\end{align*}
it is clear that $n s_{n - 1} + (q - 1) (n + r) s_n = T_n(0)$. Next, the $q$-polynomial sequence \linebreak ${\left((q - 1) r S_k(X)\right)}_{k \in \N_0}$ is clearly Carlitz-type with associated real sequence ${\left((q - 1) r a_k\right)}_k = {\left(\frac{r}{[k + r]}\right)}_k$. Besides, according to Item 4 of Proposition \ref{p1}, the $q$-polynomial sequence \linebreak ${\big(k \left(S_{k - 1}(X) + (q - 1) S_k(X)\right)\big)}_{k \in \N_0}$ is also Carlitz-type with associated real sequence \linebreak ${\left((q - 1) k a_k\right)}_k = {\left(\frac{k}{[k + r]}\right)}_k$. Consequently, the $q$-polynomial sequence ${\left(T_k(X)\right)}_{k \in \N_0}$ is Carlitz-type with associated real sequence ${\left(\frac{r}{[k + r]}\right)}_k + {\left(\frac{k}{[k + r]}\right)}_k = {\left(\frac{k + r}{[k + r]}\right)}_k$. Thus ${\left(T_k(X)\right)}_k$ coincides with ${\left(\beta_k^{(r)}(X)\right)}_k$, that is $T_k(X) = \beta_k^{(r)}(X)$ ($\forall k \in \N_0$). In particular, we have $n s_{n - 1} + (q - 1) (n + r) s_n = T_n(0) = \beta_n^{(r)}(0) = \beta_n^{(r)}$. Substituting this into \eqref{eq37} yields the required formula. This completes the proof. 
\end{proof}

From Theorem \ref{t7}, we derive the following curious corollary:

\begin{coll}\label{c6}
For all $n \in \N_0$, we have
$$
\beta_n + q \beta_{n + 1} = \sum_{k = 0}^{n} (-1)^k \frac{[k]!}{[k + 2]} S_q(n , k) .
$$
\end{coll}

\begin{proof}
Let $n \in \N_0$ be fixed. According to Formula \eqref{eq24}, we have
$$
\beta_n^{(2)} = \beta_n^{(1)} + (q - 1) \beta_{n + 1}^{(1)} = \beta_n + (q - 1) \beta_{n + 1} ,
$$
which gives $\beta_{n + 1} = \frac{\beta_n^{(2)} - \beta_n}{q - 1}$, and then
$$
\beta_n + q \beta_{n + 1} = \beta_n + q \left(\frac{\beta_n^{(2)} - \beta_n}{q - 1}\right) = \frac{q \beta_n^{(2)} - \beta_n^{(1)}}{q - 1} .
$$
It follows by using Theorem \ref{t7} that
\begin{equation}\label{eq38}
\beta_n + q \beta_{n + 1} = \sum_{k = 0}^{n} \dfrac{q \varphi_2(q , k) - \varphi_1(q , k)}{q - 1} S_q(n , k) .
\end{equation}
Further, from the second expression of $\varphi_r(q , k)$ in Theorem \ref{t7}, we have that
$$
\varphi_1(q , k) = (-1)^k \frac{[k]!}{[k + 1]} ~~\text{and}~~ \varphi_2(q , k) = (-1)^k [k]! q^{-1} \left(\frac{1}{[k + 1]} + \frac{q - 1}{[k + 2]}\right) .
$$
Thus
$$
\frac{q \varphi_2(q , k) - \varphi_1(q , k)}{q - 1} = (-1)^k \frac{[k]!}{[k + 2]} .
$$
Substituting this into \eqref{eq38} yields the required formula. This achieves the proof.
\end{proof}

Letting $q \to 1$ in Corollary \ref{c6} gives the following result on the classical Bernoulli numbers, which is very recently pointed out by the author \cite[Page 12, Formula (2.3)]{far}:

\begin{coll}\label{c7}
For all $n \in \N_0$, we have
\begin{equation}
B_n + B_{n + 1} = \sum_{k = 0}^{n} (-1)^k \frac{k!}{k + 2} S(n , k) . \tag*{\qedsymbol}
\end{equation}
\end{coll}

\rhead{\textcolor{OrangeRed3}{\it References}}

\end{document}